\documentclass[a4paper,11pt]{amsart}

\usepackage{amsfonts,amssymb,mathtools}
\usepackage{textcomp, color}

\theoremstyle{plain}

\newtheorem*{thmA}{Theorem A}
\newtheorem*{thmB}{Theorem B}

\newtheorem{thm}{Theorem}[section]
\newtheorem{lem}[thm]{Lemma}

\theoremstyle{definition}

\newtheorem{dfn}[thm]{Definition}

\newcommand{\F}{\mathbb{F}}
\newcommand{\Z}{\mathbb{Z}}

\newcommand{\C}{\mathbb{C}}
\renewcommand{\P}{\mathbb{P}}

\newcommand{\No}{\mathcal{N}}

\newcommand{\WW}{\mathcal{W}}

\DeclareMathOperator{\Cl}{Cl}

\DeclareMathOperator{\Ker}{Ker}

\DeclarePairedDelimiter\ceil{\lceil}{\rceil}
\DeclarePairedDelimiter\floor{\lfloor}{\rfloor}

\begin{document}

\title{Beauville structures in $p$-central quotients}

\author[\c{S}.\ G\"ul]{\c{S}\"ukran G\"ul}
\address{Department of Mathematics\\ Middle East Technical University\\
06800 Ankara, Turkey}
\email{gsukran@metu.edu.tr}

\keywords{Beauville groups; free group; free product; $p$-central quotients \vspace{3pt}}

\thanks{The author is supported by the Spanish Government, grant MTM2014-53810-C2-2-P, the Basque Government, grants IT753-13 and IT974-16, and T\"UB\.{I}TAK-B\.{I}DEB-2214/A}.

\begin{abstract}
We prove a conjecture of Boston that if $p\geq 5$, all $p$-central quotients of the free group on two generators and of the free product of two cyclic groups of order $p$ are Beauville groups. In the case of the free product, we also determine Beauville structures in $p$-central quotients when $p=3$. As a consequence, we give an explicit infinite family of Beauville $3$-groups, which is different from the only one that was known up to date.
\end{abstract}

\maketitle

\section{Introduction}

A \emph{Beauville surface\/} of unmixed type is a  compact complex surface isomorphic to 
$(C_1\times C_2)/G$, where $C_1$ and $C_2$ are algebraic curves of genus at least $2$ and $G$ is a finite group acting freely on $C_1\times C_2$ and faithfully on the factors $C_i$ such that $C_i/G\cong \P_1(\C)$ and the covering map $C_i\rightarrow C_i/G$ is ramified over three points for $i=1,2$.
Then the group $G$ is said to be a \emph{Beauville group\/}.

It is easy to formulate the condition for a finite group $G$ to be a Beauville group in purely group-theoretical terms. For a couple of elements $x,y \in G$, we define
\[
\Sigma(x,y)
=
\bigcup_{g\in G} \,
\Big( \langle x \rangle^g \cup \langle y \rangle^g \cup \langle xy \rangle^g \Big),
\]
that is, the union of all subgroups of $G$ which are conjugate to $\langle x \rangle$, to 
$\langle y \rangle$ or to $\langle xy \rangle$. Then $G$ is a Beauville group if and only if the following conditions hold:
\begin{enumerate}
	\item $G$ is a $2$-generator group.
	\item There exists a pair of generating sets $\{x_1,y_1\}$ and $\{x_2,y_2\}$ of $G$ such that 
$\Sigma(x_1,y_1) \cap \Sigma(x_2,y_2)=1$.
\end{enumerate}
Then $ \{x_1,y_1\}$ and $\{x_2,y_2\}$ are said to form a \emph{Beauville structure\/} for $G$.

In 2000, Catanese \cite{cat} proved that a finite abelian group is a Beauville
group if and only if it is isomorphic to $C_n \times C_n$, where $\gcd(n,6)=1$. On the other hand, all finite quasisimple groups other than $A_5$ and $SL_2(5)$ are Beauville groups \cite{FMP,FMPc} (see also  \cite{GLL} and \cite{GM}).

If $p$ is a prime, Barker, Boston and Fairbairn
\cite{BBF} have shown that the smallest non-abelian Beauville $p$-groups for $p=2$ and $p=3$ are of order $2^7$ and $3^5$, respectively. They have also proved that there are non-abelian Beauville $p$-groups of order $p^n$ for every $p\geq 5$ and every $n \geq 3$. On the other hand, in \cite{BBPV}, it has been shown that there are Beauville $2$-groups of arbitrarily high order. The existence of infinitely many Beauville $3$-groups has been settled in the affirmative in \cite{SV} and \cite{GJ}; however, these results do not yield explicit groups. Fern\'andez-Alcober and G\"ul \cite{FG} have recently given the first explicit infinite family of Beauville $3$-groups, by considering quotients of the Nottingham group over $\F_3$.

In \cite{bos}, Boston conjectured that if $p\geq5$ and $F$ is either the free group on two generators or the free product of two cyclic groups of order $p$, then its $p$-central quotients $F/\lambda_n(F)$ are Beauville groups. In this paper we prove Boston's conjecture. In fact, in the case of the free product, we extend the result to $p=3$.

The main results of this paper are as follows.

\begin{thmA}
Let $F=\langle x,y \rangle$ be the free group on two generators. Then a $p$-central quotient $F/\lambda_n(F)$ is a Beauville group if and only if $p\geq 5$ and $n\geq 2$.
\end{thmA}

\begin{thmB}
Let $F=\langle x,y \mid x^p , y^p \rangle$ be the free product of two cyclic groups of order $p$. Then a $p$-central quotient $F/\lambda_n(F)$ is a Beauville group if and only if $p\geq 5$ and $n\geq 2$ or $p=3$ and $n\geq 4$.
\end{thmB}

As a corollary of Theorem B, we give Beauville groups of order $3^n$ for every $n\geq 5$. We will see later in Theorem \ref{non-coincidence} that this infinite family only coincides at the group of order $3^5$ with the explicit infinite family of Beauville $3$-groups in \cite{FG}.

\vspace{10pt}

\noindent
\textit{Notation.\/}
We use standard notation in group theory. If $G$ is a group, then we denote by $\Cl_G(x)$ the conjugacy class of the element $x\in G$. Also, if $p$ is a prime, then we write $G^{p^i}$ for the subgroup generated by all powers $g^{p^i}$ as $g$ runs over $G$ and $\Omega_i(G)$ for the subgroup generated by the elements of $G$ of order at most $p^i$. The exponent of $G$, denoted by $\exp G$, is the maximum of the orders of all elements of $G$.

\section{The free group on two generators}

In this section, we give the proof of Theorem A. We begin by recalling the definition of $p$-central series for the convenience of the reader.

\begin{dfn}
For any group $G$, the normal series
\[
G=\lambda_1(G) \geq \lambda_2(G)\geq \dots \geq \lambda_n(G)\geq \dots
\]
given by $\lambda_n(G)=[\lambda_{n-1}(G),G]\lambda_{n-1}(G)^p$ for $n>1$ is called the \emph{$p$-central series} of $G$.
\end{dfn}
Then a quotient group $G/\lambda_n(G)$ is said to be a \emph{$p$-central quotient\/} of $G$.
To prove the main theorems, we need the following properties of the subgroups $\lambda_n(G)$ (see \cite{HB}, Definition 1.4 and Theorem 1.8, respectively): we have
\begin{equation}
\label{alternative def}
\lambda_n(G)= \gamma_1(G)^{p^{n-1}}\gamma_2(G)^{p^{n-2}}\dots \gamma_n(G),
\end{equation}
and any element of $\lambda_n(G)$ can be written in the form
\begin{equation}
\label{form of element}
a_1^{p^{n-1}}a_2^{p^{n-2}}\dots a_n \quad
\text{for some}
\quad
a_i\in \gamma_i(G).
\end{equation}

Also observe that if $\exp G/G'=p$ then $\lambda_n(G)=\gamma_n(G)$, since for any $i,j\geq1$ we have $\gamma_i(G)^{p^j}\leq \gamma_{i+j}(G)$.

\begin{lem}
\label{pn-2 powers}
Let $G$ be a group and $x,y \in G$. For $n\geq 3$,  we have
\[
(xy)^{p^{n-2}} \equiv x^{p^{n-2}}y^{p^{n-2}} \pmod { \lambda_n(G)}.
\]
\end{lem}

\begin{proof}
By the  Hall-Petrescu formula (see \cite{HB}, Lemma 1.1), we have
\begin{equation}
\label{hall-petrescu}
(xy)^{p^{n-2}} \equiv x^{p^{n-2}}y^{p^{n-2}}\pmod{\gamma_2(G)^{p^{n-2}}\prod_{r=1}^{n-2} 
\gamma_{p^r}(G)^{p^{n-2-r}}}.
\end{equation}
Now the result follows, since by (\ref{alternative def}), $\gamma_2(G)^{p^{n-2}}\leq \lambda_n(G)$ and for $1\leq r \leq n-2$ we have
\[
\gamma_{p^r}(G)^{p^{n-2-r}}\leq \lambda_{p^{r}+n-2-r}(G) \leq \lambda_n(G).
\]

\end{proof}
Note that if $y\in \lambda_2(G)$ in Lemma \ref{pn-2 powers}, then
\begin{equation}
\label{elements in same coset}
(xy)^{p^{n-2}} \equiv x^{p^{n-2}} \pmod{\lambda_n(G)}.
\end{equation}

\vspace{0.3cm}
Before we proceed to prove Theorem A, we will need to introduce a lemma.

Let $F=\langle x,y\rangle$ be the free group on two generators. Notice that for $n\geq 2$,
$\Phi(F/\lambda_n(F))$ coincides with $\lambda_2(F)/\lambda_n(F)$, and thus elements outside $\lambda_2(F)$ are potential generators in $F/\lambda_n(F)$. In order to determine Beauville structures in the quotients $F/\lambda_n(F)$, it is fundamental to control $p^{n-2}$nd powers of elements outside $\lambda_2(F)$ in these quotients groups.

\begin{lem}
\label{linearly independent}
Let $F=\langle x,y \rangle$ be the free group on two generators. Then $x^{p^{n-2}}$ and $y^{p^{n-2}}$ are linearly independent modulo $\lambda_n(F)$ for $n\geq 2$. 
\end{lem}

\begin{proof}
We argue by way of contradiction. Suppose that  $y^{ip^{n-2}} \equiv x^{p^{n-2}}\pmod{\lambda_n(F)}$. It follows from (\ref{form of element}) that $x^{-p^{n-2}}y^{ip^{n-2}}=a_1^{p^{n-1}}a_2^{p^{n-2}}\dots a_n $ for some $a_j \in \gamma_j(F)$, and then we have $y^{-ip^{n-2}}x^{p^{n-2}}a_1^{p^{n-1}} \in \gamma_2(F)$. Write $a_1= x^ky^lz$ for some $z\in \gamma_2(F)$ and some $k,l \in \Z$. Then $x^{p^{n-2}(1+kp)}y^{p^{n-2}(lp-i)}\in \gamma_2(F)$. On the other hand, an element of the free group $F$ belongs to $\gamma_2(F)$ if and only if  the exponent sum of both generators is zero. Hence we get $p^{n-2}(1+kp)=0$, which is a contradiction.
\end{proof}

As a consequence of Lemma \ref{linearly independent}, $x$ and $y$ have order $p^{n-1}$ modulo $\lambda_n(F)$.

By (\ref{elements in same coset}), if we want to know $p^{n-2}$nd powers of all elements outside $\lambda_2(F)$ in $F/\lambda_n(F)$, it is enough to know the power of each element in the set $\{y,xy^i \mid 0\leq i \leq p-1 \}$. Also, by Lemma \ref{pn-2 powers}, we have
\[
(xy^i)^{p^{n-2}} \equiv x^{p^{n-2}}y^{ip^{n-2}} \pmod{ \lambda_n(F)}
\quad
\text{for}
\quad
1\leq i \leq p-1,
\]
and since $x^{p^{n-2}}$ and $y^{p^{n-2}}$ are linearly independent modulo $\lambda_n(F)$ by Lemma \ref{linearly independent}, the following lemma is straightforward.

\begin{lem}
\label{pn-2th powers}
If $G=F/\lambda_n(F)$, the power subgroups $M^{p^{n-2}}$ are all different and of order $p$ in 
$\lambda_{n-1}(F)/\lambda_n(F)$, as  $M$ runs over the $p+1$ maximal subgroups of $G$. In particular, all elements in $M\smallsetminus \Phi(G)$ are of order $p^{n-1}$.
\end{lem}

After these preliminaries, we can now prove Theorem A.

\begin{thm}
A $p$-central quotient $F/\lambda_n(F)$ is a Beauville group  if and only if $p\geq 5$ and $n\geq 2$.
\end{thm}

\begin{proof}
For simplicity let us call $G$ the quotient group $F/\lambda_n(F)$. We first show that if $p=2$ or $3$, then $G$ is not a Beauville group. By way of contradiction, suppose that $\{u_1,v_1\}$ and $\{u_2,v_2\}$ form a Beauville structure for $G$. Since $G$ has $p+1\leq 4$ maximal subgroups, we may assume that $u_1$ and $u_2$ are in the same maximal subgroup. Then by (\ref{elements in same coset}), we have
$\langle u_1^{p^{n-2}} \rangle=\langle u_2^{p^{n-2}} \rangle$, which is a contradiction.

Thus we assume that $p\geq 5$. First of all, notice that if $n=2$, $G\cong C_p \times C_p$ is a Beauville group, by Catanese's criterion. So we will deal with the case $n\geq3$. Let $u$ and $v$ be the images in $G$ of $x$ and $y$, respectively. We claim that $\{u,v\}$ and $\{uv^2, uv^4\}$ form a Beauville structure for $G$. If $A=\{u,v,uv\}$ and $B=\{uv^2,uv^4,uv^2uv^4\}$, we need to show that
\begin{equation}
\label{check}
\langle a^g\rangle \cap \langle b^h\rangle=1,
\end{equation}
for all $a\in A$, $b\in B$, and $g,h \in G$. Observe that $a^g$ and $b^h$ lie in different maximal subgroups of $G$ in every case, since $u$ and $v$ are linearly independent modulo $\Phi(G)$ and $p\geq 5$.

Now, all elements  $a\in A$ and $b\in B$ are of order $p^{n-1}$, by Lemma \ref{pn-2th powers}. If (\ref{check}) does not hold, then
\[
\langle (a^g)^{p^{n-2}}\rangle=\langle (b^h)^{p^{n-2}}\rangle,
\]
and again by Lemma \ref{pn-2th powers}, $a^g$ and $b^h$ lie in the same maximal subgroup of $G$, which is a contradiction. We thus complete the proof that $G$ is a Beauville group.
\end{proof}

\section{ The free product of two cyclic groups of order $p$}
Now we focus on the free product $F=\langle x,y \mid x^p,y^p \rangle$ of two cyclic groups of order $p$. Notice that since $F/F'$ has exponent $p$, we have $\lambda_n(F)=\gamma_n(F)$ for all $n\geq1$.

We start with an easy lemma whose proof is left to the reader.

\begin{lem}
\label{homomorphism}
Let $\psi \colon  G_1 \to G_2$ be a group homomorphism, let $x_1, y_1 \in G_1$ and $x_2= \psi(x_1)$, $y_2=\psi(y_1)$. If $o(x_1)=o(x_2)$ then the condition $\langle x_2^{\psi(g)} \rangle \cap \langle y_2^{\psi(h)} \rangle= 1$ implies that $\langle x_1^{g} \rangle \cap \langle y_1^{h} \rangle= 1$ for  $g,h \in G_1$.
\end{lem}

To prove the main theorem we also need a result of Easterfield \cite{eas} regarding the exponent of $\Omega_i(G)$. More precisely, if $G$ is a $p$-group, then for every $i,k\geq 1$, the condition $\gamma_{k(p-1)+1}(G)=1$ implies that
\begin{equation}
\label{exp of omega}
 \exp \Omega_i(G)\leq p^{i+k-1}.
\end{equation}

\vspace{0.2cm}

A key ingredient of the proof of Theorem B will be based on $p$-groups of maximal class with some specific properties. Let $G=\langle s \rangle \ltimes A$ where $s$ is of order $p$ and $A\cong \Z_p^{p-1}$. The action of $s$ on $A$ is via $\theta$, where $\theta$ is defined by the companion matrix of the $p$th cyclotomic polynomial $x^{p-1}+\dots +x+1$. Then $G$ is the only infinite pro-$p$ group of maximal class. Since $s^p=1$ and $\theta^{p-1}+\dots +\theta+1$ annihilates $A$, this implies that for every $a\in A$,
\[
(sa)^p
=
s^p a^{s^{p-1}+\dots +s+1}
=
1.
\]
Thus all elements in $G\smallsetminus A$ are of order $p$. An alternative construction of $G$ can be given by using the ring of cyclotomic integers (see Example 7.4.14 \cite{LM}).

Let $P$ be a finite quotient of $G$ of order $p^n$ for $n\geq 3$. Let us call $P_1$ the abelian maximal subgroup of $P$ and $P_i=[P_1, P, \overset{i-1}{\ldots}, P]=\gamma_i(P)$ for $i\geq 2$. Then one can easily check that 
$\exp P_i= p^ {\ceil*{\frac{n-i}{p-1}} }$ and every element in $P_i\smallsetminus P_{i+1}$ is of order $p^ {\ceil*{\frac{n-i}{p-1}} }$.

Now we can begin to determine which $p$-central quotients of $F$ are Beauville groups. We first assume that $p=2$. The free product $F$ of two cyclic groups of order $2$ is the infinite dihedral group $D_{\infty}$. Then by Lemma 3.7 in \cite{BCG}, no finite quotient of $F$ is a Beauville group.

We next deal separately with the cases $p\geq 5$ and $p=3$.

\begin{thm}
\label{p>3}
If $p\geq 5$ then the $p$-central quotient $F/\lambda_n(F)$ is a Beauville group for every $n\geq 2$.
\end{thm}

\begin{proof}
For simplicity let us call $G$ the quotient group $F/\lambda_n(F)$. Observe that $\Omega_1(G)=G$.

If $n=2$ then $G\cong C_p\times C_p$ is a Beauville group, by Catanese's criterion. Thus we assume that $n\geq 3$. Let $P$ be the $p$-group of maximal class of order $p^n$ which is mentioned above and let $s\in P\smallsetminus P_1$ and $s_1\in P_1\smallsetminus P'$. Since all elements in $P \smallsetminus P_1$ are of order $p$ and $\lambda_n(P)=1$, the map
\begin{align*} 
 \psi \colon G & \longrightarrow P \\ 
 u& \longmapsto s^{-1}\\
 v& \longmapsto ss_1,
\end{align*}
where $u$ and $v$ are the images of $x$ and $y$, 
is well-defined and an epimorphism.

Set $k=\ceil*{\frac{n-1}{p-1}}$. Since 
$\psi(uv)=s_1$, we have $o(uv) \geq o(s_1)=p^k$. On the other hand, $\gamma_{k(p-1)+1}(G)\leq \gamma_n(G)=1$. Then by (\ref{exp of omega}), we get $\exp G \leq p^k$, and consequently
$o(uv)=p^k$. 

We are now ready to show that $G$ is a Beauville group. We claim that $\{u,v\}$ and $\{uv^2,uv^4\}$ form a Beauville structure for $G$. Let $A=\{u,v,uv\}$ and $B=\{uv^2,uv^4,uv^2uv^4\}$. Assume first that $a=u$ or $v$, which are elements of order $p$, and $b\in B$. If $\langle a^g \rangle \cap \langle b^h \rangle \neq 1$ for some $g,h \in G$, then $\langle a^g \rangle \subseteq \langle b^h \rangle$, and hence $\langle a\Phi(G) \rangle=\langle b\Phi(G) \rangle$, which is a contradiction since $p\geq 5$. Next we assume that $a=uv$. Since $p\geq 5$, for every $b \in B$ we have $\psi(b) \in P\smallsetminus P_1$, which is of order $p$. Thus for all $g,h\in G$ we have
$\langle s_1^{\psi(g)}\rangle \cap \langle \psi(b)^{\psi(h)}\rangle=1$. Since $o(uv)=o(s_1)$, it then follows from Lemma \ref{homomorphism} that $\langle a^g\rangle \cap \langle b^h\rangle=1$. This completes the proof.
\end{proof}

In order to deal with the prime $3$, we need the following lemmas.

\begin{lem}
\label{non-conjugate}
Let $G$ be a $p$-group which is not of maximal class such that $d(G)=2$. Then for every $x\in G$ there exists $t\in \Phi(G)\smallsetminus \{[x,g] \mid g\in G\}$.
\end{lem}

\begin{proof}
Note that a $p$-group has maximal class if and only if it has an element with centralizer of order $p^2$ (see \cite[III.14.23]{hup}). Thus for every $x\in G$ we have $|C_G(x)|\geq p^3$, and hence
\[
|\{[x,g] \mid g\in G\}|= |\Cl_G(x)|=|G:C_G(x)|\leq p^{n-3}.
\]
Since $|\Phi(G)|=p^{n-2}$ there exists $t\in \Phi(G)$ such that $t\notin \{[x,g] \mid g\in G\}$.
\end{proof}

\begin{lem}\textup{\cite[Lemma~3.8]{FG} }
\label{intersection}
Let $G$ be a finite $p$-group and let $x \in G \smallsetminus \Phi(G)$ be an element of order $p$.
If $t \in \Phi(G)\smallsetminus \{[x,g] \mid g \in G \}$ then 
\[
\Big(\bigcup_{g\in G} {\langle x\rangle}^g  \Big)
\bigcap
\Big(\bigcup_{g\in G} {\langle xt\rangle}^g \Big)= 1.
\]
\end{lem}

\begin{thm}
\label{p=3}
Let $p=3$. Then the follıowing hold.
\begin{enumerate}
	\item The $p$-central quotient $F/\lambda_n(F)$ is a Beauville group if and only if $n\geq 4$.
	\item The series $\{\lambda_n(F)\}_{n\geq 4}$ can be refined to a normal series of $F$ such that two consecutive terms of the series have index $p$ and for every term $N$ of the series $F/N$ is a Beauville group.
\end{enumerate}
\end{thm}

\begin{proof}
Since the smallest Beauville $3$-group is of order $3^5$, the quotient  $F/\lambda_n(F)$ can only be a Beauville group if $n\geq 4$. We first assume that $n=4$. Now consider the group
\[
H= \langle a,b,c,d,e \mid a^3=b^3=c^3=d^3=e^3=1, [b,a]=c, [c,a]=d, [c,b]=e \rangle,
\]
where we have omitted all commutators between generators which are trivial. This is the smallest Beauville $3$-group. Since $\lambda_4(H)=1$, $F/\lambda_4(F)$ maps onto $H$. On the other hand, it is clear that $|F/\lambda_4(F)|\leq 3^5$ and so  $F/\lambda_4(F)\cong H$. Consequently, $F/\lambda_4(F)$ is a Beauville group. Thus we assume that $n\geq5$.

Now let us call $G$ the quotient group $F/\lambda_n(F)$. Consider the map $\psi \colon G\longrightarrow P$ defined in the proof of Theorem \ref{p>3}. Since $\psi$ is an epimorphism, we have $\psi(\lambda_{n-1}(G))=\lambda_{n-1}(P)$. Observe that the subgroup $\Ker \psi \cap \lambda_{n-1}(G)$ has index $3$ in $\lambda_{n-1}(G)$, since $\lambda_{n-1}(P)$ is of order $3$ . Choose a normal subgroup $N$ of $F$ such that $\lambda_{n}(F)\leq N < \lambda_{n-1}(F)$ and $N/\lambda_n(F)\leq \Ker \psi$. Then $\psi$ induces an epimorphism $\overline{\psi}$ from $F/N$ to $P$.

 We will see that $L=F/N$ is a Beauville group, which simultaneously proves (i) and (ii). Let $u$ and $v$ be the images of $x$ and $y$ in $L$, respectively. Set $k={\ceil*{\frac{n-1}{2}}}$. Then $o(uv)\leq o(xy\lambda_n(F))=3^k$. On the other hand, since $\overline{\psi}(uv)=s_1$, we have $o(uv)\geq o(s_1)= 3^k$, and consequently we get $o(uv)=3^k$ in $L$. Since $F/\lambda_4(F)\cong H$ is not of maximal class, $L$ is not of maximal class. Thus, by Lemma \ref{non-conjugate}, there exist $z, t \in \Phi(L)$ such that  $z \notin \{ [u,l] \mid l\in L\}$ and $t \notin \{[v,l] \mid l\in L\}$. We claim that $\{u,v\}$ and $\{(uz)^{-1}, vt\}$ form a Beauville structure for $L$. Let $A=\{u,v,uv\}$ and $B=\{(uz)^{-1}, vt,  (uz)^{-1}vt\}$.

If $a=u$, which is of order $3$, and $b=vt$ or $(uz)^{-1}vt$, then we get $\langle a^g \rangle \cap \langle b^h\rangle=1$ for every $g,h\in L$, as in the proof of Theorem \ref{p>3}. When $a=v$ and $b=(uz)^{-1}$ or $(uz)^{-1}vt$, the same argument applies. If we are in one of the following cases: $a=u$ and $b=(uz)^{-1}$, or $a=v$ and $b=vt$, then the  condition 
$\langle a^g \rangle \cap \langle b^h\rangle=1$ follows from Lemma \ref{intersection}.

It remains to check the case when $a=uv$ and $b\in B$. For every $b\in B$, we have $\overline{\psi}(b) \in P\smallsetminus P_1$, which has order $3$. Since $o(uv)=o(s_1)$, the condition $\langle a^g \rangle \cap \langle b^h\rangle=1$  follows from Lemma \ref{homomorphism}, as in the proof of Theorem \ref{p>3}. This completes the proof.
\end{proof}

Thus the quotients in Theorem \ref{p=3} constitute an infinite family of Beauville $3$-groups of order $3^n$ for all $n\geq 5$.

In \cite{FG}, it was given the first explicit infinite family of Beauville $3$-groups, by considering quotients of the Nottingham group over $\F_3$. We will show that these two infinite families of Beauville $3$-groups only coincide at the group of order $3^5$.

Before proceeding we recall the definition of the Nottingham group and some of its properties. The \emph{Nottingham group\/} $\No$ over the field $\F_p$, for odd $p$, is the (topological) group of normalised automorphisms of the ring $\F_p[[t]]$ of formal power series. For any positive integer $k$, the automorphisms $f\in\No$ such that $f(t)=t+\sum_{i\ge k+1} \, a_it^i$ form an open normal subgroup $\No_k$ of $\No$ of index $p^{k-1}$. The lower central series of $\No$ is given by
\begin{equation}
\label{lcs}
\gamma_i(\No)= \No_{r(i)},
\quad
\text{where}
\quad
 r(i)=i+1+\floor*{\frac{i-2}{p-1}},
\end{equation}
and 
\begin{equation}
\label{power subgroup}
\No_k^p=\No_{kp+r},
\quad
\text{where}
\quad
0\leq r\leq p-1
\quad
\text{is the residue of $k$ modulo $p$}
\end{equation}
(see \cite{cam}, Remark 1 and Theorem 6, respectively).

 Also, each non-trivial normal subgroup of $\No$ lies between some $\gamma_i(\No)$ and $\gamma_{i+1}(\No)$ (see \cite {cam}, Remark 1 and Proposition 2).

By Theorem 3.10 in \cite{FG}, if $p=3$ then a quotient $\No/\No_k$ is a Beauville group if and only if $k\geq 6$ and $k\neq z_m$ for all $m\geq 1$, where $z_m= p^m+p^{m-1}+\dots+p+2$. Furthermore, by Theorem 3.11 in \cite{FG}, for $i\geq1$ there exists a normal subgroup $\WW$ between $\No_{ip+3}$ and $\No_{ip+1}$ such that $\No/\WW$ is a Beauville group. This gives quotients of $\No$ which are Beauville groups of every order $3^n$ with $n\geq5$.

\begin{thm}
	\label{non-coincidence}
 Let $N\neq \gamma_4(F)$ be a normal subgroup of $F$ such that $F/N$ is a Beauville group. Then $F/N$ is not isomorphic to any quotient of $\No$ which is a Beauville group. On the other hand, $F/\gamma_4(F)$ is isomorphic to $\No/ \gamma_4(\No)$. 
\end{thm}

\begin{proof}
	Since there is only one Beauville group of order $3^5$ \cite{BBF}, $F/\gamma_4(F)$ is necessarily isomorphic to $\No/\gamma_4(\No)$. Now suppose that $F/N\cong \No/ \WW$ where $\gamma_n(F)\leq N<\gamma_{n-1}(F)$ for $n\geq5$ and $F/N$ is a Beauville group. Since $F/N$ is of class $n-1$ and $\WW$ lies between two consecutive terms of the lower central series, we have $\gamma_n(\No)\leq \WW < \gamma_{n-1}(\No)$. Note that if $n=5$ then $\No_7=\gamma_5(\No)\leq \WW < \gamma_4(\No)=\No_6$ and so $\WW=\gamma_5(\No)$. If $n>5$ then $\WW \leq \gamma_5(\No)$. Consequently the isomorphism  $F/N \cong \No/\WW$ implies that $F/\gamma_5(F)N \cong \No/ \gamma_5(\No)$. We next show that this is not possible.
	
	Note that by (\ref{lcs}), we have $\gamma_2(\No)=\No_3$ and by (\ref{power subgroup}), $\No_3^3=\No_9$. Thus the exponent of $\gamma_2(\No/\gamma_5(\No))$ is $3$. On the other hand, as in the proof of Theorem \ref{p=3}, there is an epimorphism from $F/\gamma_5(F)N$ to a $p$-group of maximal class $P$ of order $3^5$ with $\exp P'=3^2$. Thus $\No/\gamma_5(\No)$ cannot be isomorphic to $F /\gamma_5(F)N$.
	
\end{proof}

\section*{Acknowledgments}
I would like to thank G. Fern\'andez-Alcober, N. Gavioli and C. M. Scoppola for helpful comments and suggestions. Also, I would like to thank the Department of Mathematics at the University of the Basque Country for its hospitality while this paper was being written.

\end{document}